\documentclass[11pt]{amsart}

\usepackage{amsmath,amssymb,amsthm,enumerate}

\DeclareMathOperator{\Char}{char}

\DeclareMathOperator{\lt}{lt}
\DeclareMathOperator{\lm}{lm}
\DeclareMathOperator{\HP}{HP}
\DeclareMathOperator{\HF}{HF}

\newcommand{\sm}{\setminus}

\newcommand{\C}{\mathbb{C}}

\newcommand{\N}{\mathbb{N}}
\newcommand{\Z}{\mathbb{Z}}
\newcommand{\Q}{\mathbb{Q}}
\newcommand{\R}{\mathbb{R}}

\newcommand{\F}{\mathbb{F}}

\newcommand{\E}{\mathbb{E}}
\newcommand{\A}{\mathbb{A}}

\newcommand{\HH}{\mathcal{H}}
\newcommand{\ve}{\varepsilon}

\newcommand{\1}{^{-1}}

\newcommand{\f}[2]{\frac{#1}{#2}}

\newtheorem{theorem}{Theorem}[section]

\newtheorem{definition}[theorem]{Definition}

\newtheorem{corollary}[theorem]{Corollary}
\newtheorem{proposition}[theorem]{Proposition}
\newtheorem{lemma}[theorem]{Lemma}

\numberwithin{theorem}{section}

\title{Three-term polynomial progressions in subsets of finite fields}
\author{Sarah Peluse}
\address{Department of Mathematics, Stanford University, Stanford, California 94305}
\email{speluse@stanford.edu}

\begin{document}

\begin{abstract}
Bourgain and Chang recently showed that any subset of $\F_p$ of density $\gg p^{-1/15}$ contains a nontrivial progression $x,x+y,x+y^2$. We answer a question of theirs by proving that if $P_1,P_2\in\Z[y]$ are linearly independent and satisfy $P_1(0)=P_2(0)=0$, then any subset of $\F_p$ of density $\gg_{P_1,P_2}p^{-1/24}$ contains a nontrivial polynomial progression $x,x+P_1(y),x+P_2(y)$.
\end{abstract}

\maketitle

\section{Introduction}
Let $P_1,\dots,P_m\in\Z[y]$ be polynomials satisfying $P_1(0)=\dots=P_m(0)=0$, and for each $N\in\N$, let $[N]$ denote the set $\{1,\dots,N\}$. Bergelson and Leibman's polynomial generalization of Szemer\'edi's Theorem~\cite{BL} states that if $A\subset[N]$ contains no progression
\begin{equation}\label{polyprog}
x,x+P_1(y),\dots,x+P_m(y)
\end{equation}
with $y\neq 0$, then $|A|=o_{P_1,\dots,P_m}(N)$.

When each $P_i$ is linear, Gowers's proof of Szemer\'edi's Theorem~\cite{Go} gives the explicit bound
\[
|A|\ll_{P_1,\dots,P_m}\f{N}{(\log\log{N})^{c_m}}.
\]
Quantitative bounds are known for the size of subsets of $[N]$ lacking nontrivial polynomial progressions in only two other special cases. The case when $m=1$ is covered by S\'ark\"ozy's Theorem~\cite{S}, which dealt with $P_1=y^2$, and later generalizations to other polynomials, such as work by S\'ark\"ozy~\cite{S2}, Balog, Pelik\'an, Pintz, and Szemer\'edi~\cite{BPPS}, Slijep\v{c}evi\'c~\cite{Sl}, and Lucier~\cite{L}. When $m\geq 2$, the only quantitative result for progressions involving nonlinear polynomials is due to Prendiville~\cite{P}, who dealt with the special case when $P_i=a_iy^d$ for a fixed $d\in\N$.

In this paper, we consider the related problem of bounding the size of $A\subset\F_q$ lacking nontrivial polynomial progressions. Of course, any bounds in the integer setting automatically hold in the prime field setting. However, one should expect that superior bounds hold in finite fields as the degrees of the $P_i$'s increase. Indeed, if, for example, $\deg P_m=d$ and $P_m(y)\neq 0$ for any $y\neq 0$, then one can greedily construct a subset of $[N]$ of density $\gg N^{-1/d}$ that lacks nontrivial progressions of the form~(\ref{polyprog}), since we must have $x\in[N]$ and $y\ll N^{1/d}$ for~(\ref{polyprog}) to lie in $[N]$. In the finite field setting, in contrast, both $x$ and $y$ can clearly run over all of $\F_q$ regardless of the degrees of the $P_i$, so such a construction does not work.

We will focus on the case when $m=2$ and $P_1$ and $P_2$ are linearly independent, so that the terms of the progression
\begin{equation}\label{polyprog2}
x,x+P_1(y),x+P_2(y)
\end{equation}
satisfy no linear relation. In this situation, we can prove a power-saving bound on the size of subsets of $\F_q$ lacking a nontrivial progression of the form~(\ref{polyprog2}), provided the characteristic of $\F_q$ is large enough:
\begin{theorem}\label{main2}
Let $P_1,P_2\in\Z[y]$ be two linearly independent polynomials with $P_1(0)=P_2(0)=0$. There exists a constant $c_{P_1,P_2}>0$ depending only on $P_1$ and $P_2$ such that if the characteristic of $\F_q$ is at least $c_{P_1,P_2}$, then any $A\subset\F_q$ containing no nontrivial progression
\[
x,x+P_1(y),x+P_2(y),\text{ }y\neq 0,
\]
satisfies
\begin{equation}\label{Abound}
|A|\ll_{P_1,P_2}q^{1-1/24}.
\end{equation}
\end{theorem}
Note that the exponent of $q$ in~(\ref{Abound}) is independent of $P_1$ and $P_2$. Thus, when the degree of one of $P_1$ or $P_2$ is large enough, the conclusion of Theorem~\ref{main2} is stronger than what can possibly hold in the integer setting.

Since the number of trivial three-term polynomial progressions in $A$ is bounded above by $|A|$, Theorem~\ref{main2} is a consequence of the following result, which counts three-term polynomial progressions in subsets of finite fields:
\begin{theorem}\label{main1}
Let $P_1,P_2\in\Z[y]$ be two linearly independent polynomials satisfying $P_1(0)=P_2(0)=0$. There exists a $c_{P_1,P_2}>0$ depending only on $P_1$ and $P_2$ such that if the characteristic of $\F_q$ is at least $c_{P_1,P_2}$ and $A,B,C\subset\F_q$, then
\begin{align*}
\#\{(x,y)\in\F_q^2:(x,x+P_1(y)&,x+P_2(y))\in A\times B\times C\}= \\
&\f{|A||B||C|}{q}+O_{P_1,P_2}((|A||B||C|)^{1/2}q^{1/2-1/16}).
\end{align*}
\end{theorem}
Theorem~\ref{main1} says that if the characteristic of $\F_q$ is large enough, then any subset of $\F_q$ of density at least $q^{-1/24+\ve}$ contains very close to the expected number of progressions~(\ref{polyprog2}) in a random set of the same density.

Bourgain and Chang~\cite{BC} were the first to consider the problem of finding quantitative bounds for the polynomial Szemer\'edi Theorem in finite fields. In~\cite{BC}, Bourgain and Chang prove that
\[
\#\{(x,y)\in\F_p^2:(x,x+y,x+y^2)\in A^3\}=\f{|A|^3}{p}+O(|A|^{3/2}p^{1/2-1/10}),
\]
when $A\subset\F_p$, and ask whether such a result with a power-saving error term holds when $y$ and $y^2$ are replaced by any pair of linearly independent polynomials with zero constant term. Thus, Theorem~\ref{main1} answers their question in the affirmative. Note that the error term in Theorem~\ref{main1} is larger than the error term in Bourgain and Chang's result, however, so we do not quantitatively recover their result when $P_1=y$ and $P_2=y^2$.

Though Bourgain and Chang were the first to consider polynomial progressions, there was work prior to theirs on other nonlinear configurations in finite fields. For example, Shkredov~\cite{Sh} showed that if $A,B,C\subset\F_q$ satisfy $|A||B||C|\gg p^{5/2}$, then there exist $x,y\in\F_p$ such that $(x,x+y,xy)\in A\times B\times C$.

Results such as Shkredov's, Bourgain and Chang's, and ours are connected to questions about expanding polynomials and sum-product phenomena for large subsets of finite fields. One corollary of Shkredov's result is that if $A,B\subset\F_p$, then
\[
\#\{a^2+ab:a\in A,b\in B\}\geq p-1-\f{40p^{5/2}}{|A||B|},
\]
so that $\#\{a^2+ab:a\in A,b\in B\}\gg p$ whenever $|A||B|\gg p^{3/2}$. Similarly, Theorem~\ref{main1} implies that if $P\in\Z[y]$ has degree at least two, then the polynomial $Q(x,y):=x+P(y-x)$ satisfies
\begin{equation}\label{expand}
\#\{Q(a,b):a\in A,b\in B\}\gg_{P_1,P_2}q
\end{equation}
whenever $A,B\subset\F_q$, $|A||B|\gg_{P_1,P_2}q^{2-1/8}$, and the characteristic of $\F_q$ is large enough.

Expanding polynomials and sum-product for large sets have been studied extensively, and related results can be found in papers of Hart, Iosevich, and Solymosi~\cite{HIS}, Vu~\cite{V}, Hart, Li, and Shen~\cite{HLS}, Bukh and Tsimerman~\cite{BT}, and Tao~\cite{T}. Indeed, the bound~(\ref{expand}) is not new--it follows immediately from Theorem 1 of~\cite{T}. Work on large sets has tended to combine some sort of algebraic input (such as the Weil bound) with extensive use of the Cauchy-Schwarz inequality and Fourier analysis on $\F_q$. Our proof will also be in this vein.

A common approach to counting configurations such as~(\ref{polyprog}) in subsets of abelian groups involves bounding averages of the type
\begin{equation}\label{polyavg}
\E_{x,y}f_0(x)f_1(x+P_1(y))\dots f_m(x+P_{m}(y)),
\end{equation}
for $f_i$ with $\|f_i\|_\infty\leq 1$ and $f_m$ having mean zero. Repeated applications of Cauchy-Schwarz are often used to bound~(\ref{polyavg}) in terms of an average over some other (often much longer) configuration that is easier to deal with. When starting with non-linear polynomial configurations, usually Cauchy-Schwarz is used to replace these non-linear polynomials with their discrete derivative. This eventually leads to a bound for~(\ref{polyavg}) in terms of an average of averages over linear configurations. For example, Prendiville~\cite{P} bounds~(\ref{polyavg}) by an average of local Gowers $U^s$-norms, where the degree $s$ grows extremely quickly as the degrees of the $P_i$ and the length of the progression grow.

The general strategy of the proof of Theorem~\ref{main1} is also to  use Cauchy-Schwarz to bound
\begin{equation}\label{3polyavg}
\E_{x,y\in\F_q}f_0(x)f_1(x+P_1(y))f_2(x+P_2(y))
\end{equation}
in terms of an average
\[
\E_{x,y}f_2(x)f_2(x+Q_1(y))\dots f_2(x+Q_m(y))
\]
over some other polynomial progression. We will never use Cauchy-Schwarz to reduce the degrees of $P_1$ and $P_2$, however. Instead, we will apply Cauchy-Schwarz so that, as we range over $x$ and $y$, the $(m+1)$-tuples $(x,x+Q_1(y),\dots,x+Q_m(y))$ are close to being equidistributed in $\F_q^{m+1}$. Thus, the average~(\ref{3polyavg}) is always small whenever $f_2$ has mean zero.

This paper is organized as follows. In Section~\ref{2}, we will bound~(\ref{3polyavg}) in terms of an average over a length $2$ polynomial progression $x,x+Q_{P_1,P_2}(y)$. Here $x\in\F_q$ and $y$ ranges over the $\F_q$-points $V_{P_1,P_2}(\F_q)$ of some algebraic variety. Showing that the map $(x,y)\mapsto (x,x+Q_{P_1,P_2}(y))$ is close to equidistributed boils down to checking that $Q_{P_1,P_2}$ is sufficiently non-degenerate on $V_{P_1,P_2}$. We verify this non-degeneracy in Section~\ref{3}, and then in Section~\ref{4} complete the proof of Theorem~\ref{main1}.

\section*{Acknowledgments}
The author thanks Brian Conrad for many helpful conversations and Will Sawin, Kannan Soundararajan, and the anonymous referee for helpful comments on earlier versions of this paper.

This material is based upon work supported by the National Science Foundation under Grant No. DMS-1440140 while the author was in residence at the Mathematical Sciences Research Institute during the Spring 2017 semester. The author is also supported by the National Science Foundation Graduate Research Fellowship Program under Grant No. DGE-114747 and by the Stanford University Mayfield Graduate Fellowship.

\section{The Cauchy-Schwarz argument}\label{2}

\subsection{Notation} We will first fix notation and normalizations. For all sets $S$ and functions $f:S\to\C$, we write the average of $f$ over $S$ as
\[
\E_{x\in S}f(x):=\f{1}{|S|}\sum_{x\in S}f(x).
\]
We will often write $\E_{x_1,\dots,x_m}$ in place of $\E_{(x_1,\dots,x_m)\in\F_q^m}$ when averaging over $\F_q^m$. For $f:\F_q\to\C$, we also set
\[
\|f\|_{L^2}^2:=\E_{x\in\F_q}|f(x)|^2.
\] 

Let $\widehat{\F}_q$ denote the group of additive characters of $\F_q$ and $1$ denote the trivial character. For any $\psi\in\widehat{\F}_q$, the Fourier transform of $f$ at $\psi$ is
\[
\hat{f}(\psi):=\E_{x\in\F_q}f(x)\overline{\psi(x)}.
\]
Then we have the Fourier inversion formula:
\[
f(x)=\sum_{\psi\in\widehat{\F}_q}\hat{f}(\psi)\psi(x)
\]
and Parseval's identity:
\[
\E_xf_1(x)\overline{f_2(x)}=\sum_{\psi\in\widehat{\F}_q}\hat{f_1}(\psi)\overline{\hat{f_2}(\psi)},
\]
so that
\[
\|f\|_{L^2}^2=\sum_{\psi\in\widehat{\F}_q}|\hat{f}(\psi)|^2.
\]

\subsection{The averages $\Lambda_{P_1,P_2}$ and $\Lambda_{P_1}$}

Fix $P_1,P_2\in\Z[y]$ such that $P_1(0)=P_2(0)=0$ and $P_1$ and $P_2$ are linearly independent. We write
\[
P_1=\sum_{i=1}^{r_1}a_iy^i\text{ and }P_2=\sum_{j=1}^{r_2}b_jy^j
\]
with $a_{r_1},b_{r_2}\neq 0$ and assume, without loss of generality, that $r_2\geq r_1$. By replacing $P_1$ by $P_1-P_2$ and $P_2$ by $-P_2$ if needed, we may also assume that $a_{r_1}\neq b_{r_1}$ if $r_1=r_2$. Since $P_1$ and $P_2$ are linearly independent, we have
\[
P_2':=P_2-P_1=\sum_{k=1}^{r_2}c_ky^k
\]
with $c_{r_2}\neq 0$, and, when $r_1=r_2$,
\[
P_2=\f{b_{r_2}}{a_{r_1}}P_1+P_3,
\]
where
\[
P_3=\sum_{\ell=1}^{r_3}d_{\ell}y^\ell
\]
with $r_1>r_3>0$ and $d_{r_3}\neq 0$.

Let $\F_q$ be a finite field. For any $f_0,f_1,f_2:\F_q\to\R$, we define
\[
\Lambda_{P_1,P_2}(f_0,f_1,f_2):=\E_{x,y\in\F_q}f_0(x)f_1(x+P_1(y))f_2(x+P_2(y))
\]
and
\[
\Lambda_{P_1}(f_0,f_1):=\E_{x,y\in\F_q}f_0(x)f_1(x+P_1(y)).
\]
Note that if $A,B,C\subset\F_q$, then $\Lambda_{P_1,P_2}(1_A,1_B,1_C)$ is the normalized count of the number of polynomial progressions $(x,x+P_1(y),x+P_2(y))$ in $A\times B\times C$. Let $\alpha,\beta,$ and $\gamma $ be the densities of $A,B,$ and $C$, respectively, in $\F_q$. Setting $f_B:=1_B-\beta$ and $f_C:=1_C-\gamma$, we see that
\[
|\Lambda_{P_1,P_2}(1_A,1_B,1_C)-\alpha\beta\gamma|\leq|\Lambda_{P_1,P_2}(1_A,1_B,f_C)|+\gamma|\Lambda_{P_1}(1_A,f_B)|.
\]
Indeed, this follows from the decomposition
\[
\Lambda_{P_1,P_2}(1_A,1_B,1_C)=\Lambda_{P_1,P_2}(1_A,1_B,f_C)+\Lambda_{P_1,P_2}(1_A,f_B,\gamma)+\Lambda_{P_1,P_2}(1_A,\beta,\gamma),
\]
and the fact that $\Lambda_{P_1,P_2}(1_A,f_B,\gamma)=\gamma\Lambda_{P_1}(1_A,f_B)$ and $\Lambda_{P_1,P_2}(1_A,\beta,\gamma)=\alpha\beta\gamma$. Thus, bounds on $\Lambda_{P_1,P_2}(1_A,1_B,f_C)$ and $\Lambda_{P_1}(1_A,f_B)$ yield a bound on the difference between the actual number of three-term progressions in $A\times B\times C$ and the expected number if $A$, $B$, and $C$ were random subsets of $\F_q$ of density $\alpha$, $\beta$, and $\gamma$, respectively.

Bounding $\Lambda_{P_1}(1_A,f_B)$ is quite simple. Let $f_0,f_1:\F_q\to\R$ be any two real-valued functions with $\E_xf_1(x)=0$. By Fourier inversion, we have
\begin{equation}\label{fi}
\Lambda_{P_1}(f_0,f_1)=\sum_{\psi_0,\psi_1\in\widehat{\F}_q}\widehat{f_0}(\psi_0)\widehat{f_1}(\psi_1)[\E_{x}\psi_0(x)\psi_1(x)][\E_y\psi_1(P_1(y))].
\end{equation}

The orthogonality relation for characters says that $\E_{x}\psi_0(x)\psi_1(x)$ equals $1$ if $\psi_0=\overline{\psi_1}$ and equals $0$ otherwise, and the Weil bound says that $\E_{y}\psi_1(P_1(y))\leq(\deg{P_1})q^{-1/2}$ whenever $\psi_1\neq 1$. Thus, since $\widehat{f_1}(1)=0$, it follows that~(\ref{fi}) is bounded above by
\[
(\deg{P_1})q^{-1/2}\sum_{\psi\in\widehat{\F}_q}\widehat{f_0}(\psi)\overline{\widehat{f_1}(\psi)}.
\]
By Parseval's identity and Cauchy-Schwarz, we have
\[
\sum_{\psi\in\widehat{\F}_q}\widehat{f_0}(\psi)\overline{\widehat{f_1}(\psi)}= \E_{x}f_0(x)f_1(x)\leq \|f_0\|_{L^2}\|f_1\|_{L^2}.
\]
Thus, $|\Lambda_{P_1}(f_0,f_1)|\leq (\deg{P_1})\|f_0\|_{L^2}\|f_1\|_{L^2}q^{-1/2}$.

Now, since $\|1_A\|_{L^2}=\alpha^{1/2}$ and $\|f_B\|_{L^2}=(\beta-\beta^2)^{1/2}$, we conclude that $|\Lambda_{P_1}(1_A,f_B)|\leq (\deg{P_1})\alpha^{1/2}\beta^{1/2} q^{-1/2}$, and thus that
\[
|\Lambda_{P_1,P_2}(1_A,1_B,1_C)-\alpha\beta\gamma|\leq|\Lambda_{P_1,P_2}(1_A,1_B,f_C)|+(\deg P_1)\alpha^{1/2}\beta^{1/2}\gamma q^{-1/2}.
\]
Theorem~\ref{main1} will thus be a consequence of the following result.
\begin{theorem}\label{main}
Suppose that $f_0,f_1,f_2:\F_q\to\R$ and $\E_x f_2(x)=0$. Then
\[
\Lambda_{P_1,P_2}(f_0,f_1,f_2)\ll_{P_1,P_2}\|f_0\|_{L^2}\|f_1\|_{L^2}\|f_2\|_{L^2}q^{-1/16}.
\]
\end{theorem}
The proof of Theorem~\ref{main} proceeds by bounding $\Lambda_{P_1,P_2}(f_0,f_1,f_2)$ in terms of an average $\Lambda'_{P_1,P_2}(f_2,f_2)$ over a polynomial progression of length two,
\[
x,x+Q_{P_1,P_2}(y),
\]
similar to the average $\Lambda_{P_1}(f_0,f_1)$. The only difference is that $y$ does not vary over $\F_q$. It instead varies over the $\F_q$-points $V_{P_1,P_2}(\F_q)$ of an affine variety that depends on $P_1$ and $P_2$. We can then bound $\Lambda'_{P_1,P_2}(f_2,f_2)$ in the same manner that we bounded $\Lambda_{P_1}(f_0,f_1)$, provided we have a nontrivial bound, uniform in $q$, for the character sum
\[
\sum_{y\in V_{P_1,P_2}(\F_q)}\psi(Q_{P_1,P_2}(y))
\]
whenever $\psi\in\widehat{\F}_q\sm\{1\}$. The remainder of this section will focus on bounding $\Lambda_{P_1,P_2}(f_0,f_1,f_2)$ in terms of $\Lambda_{P_1,P_2}'(f_2,f_2)$, which we will define next.

\subsection{Bounding $\Lambda_{P_1,P_2}$}
We define polynomials $R_{P_1,P_2}^{(1)},R_{P_1,P_2}^{(2)},R_{P_1,P_2}^{(3)},R_{P_1,P_2}^{(4)}$, and $Q_{P_1,P_2}\in \Z[y_1,\dots,y_8]$ by
\[
R_{P_1,P_2}^{(1)}:=P_1(y_4)-P_1(y_3)-P_1(y_2)+P_1(y_1),
\]
\[
R_{P_1,P_2}^{(2)}:=P_1(y_8)-P_1(y_7)-P_1(y_6)+P_1(y_5),
\]
\[
R_{P_1,P_2}^{(3)}:=P_2(y_6)-P_2(y_5)-P_2(y_2)+P_2(y_1),
\]
\[
R_{P_1,P_2}^{(4)}:=P_2'(y_7)-P_2'(y_5)-P_2'(y_3)+P_2'(y_1),
\]
and
\[
Q_{P_1,P_2}:=P_2(y_8)-P_2(y_7)-P_2(y_4)+P_2(y_3).
\]
For any field $\F$, set
\[
V_{P_1,P_2}(\F):=\{y\in\F^8:R_{P_1,P_2}^{(i)}(y)=0\text{ for }i=1,2,3,4\}
\]
and
\[
W_{P_1,P_2}(\F):=\{y\in V_{P_1,P_2}(\F)^2: Q_{P_1,P_2}(y_1,\dots,y_8)=Q_{P_1,P_2}(y_9,\dots,y_{16})\}.
\]
For any $f_0,f_1:\F_q\to\R$, we define
\[
\Lambda'_{P_1,P_2}(f_0,f_1):=\E_{x\in\F_q, y\in V_{P_1,P_2}(\F_q)}f_0(x)f_1(x+Q_{P_1,P_2}(y)).
\]
The following proposition bounds $\Lambda_{P_1,P_2}(f_0,f_1,f_2)$ in terms of $\Lambda'_{P_1,P_2}(f_2,f_2)$.
\begin{proposition}\label{2pt}
Suppose that $f_0,f_1,f_2:\F_q\to\R$. Then
\[
\Lambda_{P_1,P_2}(f_0,f_1,f_2)\leq \f{|V_{P_1,P_2}(\F_q)|}{q^4}\|f_0\|_{L^2}\|f_1\|_{L^2}\|f_2\|_{L^2}^{3/4}|\Lambda'_{P_1,P_2}(f_2,f_2)|^{1/8}.
\]
\end{proposition}
We will show in Section~\ref{3} that $\dim V_{P_1,P_2}(\bar\F_q)\leq 4$ whenever the characteristic of $\F_q$ is sufficiently large. This, combined with the Lang-Weil bound~\cite{LW}, implies that $|V_{P_1,P_2}(\F_q)|\ll_{P_1,P_2}q^4$. Thus, Proposition~\ref{2pt} does indeed bound $\Lambda_{P_1,P_2}(f_0,f_1,f_2)$ in terms of $\Lambda_{P_1,P_2}'(f_2,f_2)$.

Before proving Proposition~\ref{2pt}, we will illustrate the main way in which Cauchy-Schwarz is used in the proof. Let $R_1,\dots,R_m\in\Z[y_1,\dots,y_n]$ and $S\subset\F_q^n$, and suppose that we want to bound
\begin{equation}\label{gen}
\E_{x\in\F_q,y\in S}f_0(x)f_1(x+R_1(y))\cdots f_m(x+R_m(y)),
\end{equation}
where $f_0,\dots,f_m:\F_q\to\R$ and $\|f_i\|_\infty\leq 1$ for $i=0,\dots,m$.

We can rewrite~(\ref{gen}) as follows by collecting together the elements of $S$ in each fiber of $R_1$:
\[
\f{1}{q^2}\sum_{x,z\in\F_q}f_0(x)f_1(x+z)\f{1}{|S|/q}\sum_{\substack{y\in S \\ R_1(y)=z}}f_2(x+R_2(y))\cdots f_m(x+R_m(y)).
\]
Applying Cauchy-Schwarz in the outer sum, we bound the modulus squared of the above by
\[
\f{1}{q^2}\sum_{x,z\in\F_q}\f{1}{|S|^2/q^2}\sum_{\substack{y_1,y_2\in S \\ R_1(y_1)=z \\ R_1(y_2)=z}}f_2(x+R_2(y_1))f_2(x+R_2(y_2))\cdots f_m(x+R_m(y_2)),
\]
which, summing the interior sum over $z\in\F_q$, equals
\[
\f{1}{q}\sum_{x\in\F_q}\f{1}{|S|^2/q}\sum_{\substack{y_1,y_2\in S \\ R_1(y_1)=R_1(y_2)}}f_2(x+R_2(y_1))f_2(x+R_2(y_2))\cdots f_m(x+R_m(y_2)).
\]
Note that the inner sum is the sum over $y\in S\times_{R_1}S$, where $S\times_{R_1}S=\{(y_1,y_2)\in S^2:R_1(y_1)=R_1(y_2)\}$ is the fiber product over $R_1$ of the set $S$ with itself. So if $|S\times_{R_1}S|\ll |S|^2/q$, then~(\ref{gen}) is
\[
\ll|\E_{\substack{x\in\F_q \\ y\in S\times_{R_1}S}}f_2(x+R_2(y_1))f_2(x+R_2(y_2))\cdots f_m(x+R_m(y_2))|^{1/2}.
\]
Thus, Cauchy-Schwarz can be used to bound an average over $x\in\F_q$ and $y\in S$ in terms of an average over $x\in\F_q$ and $(y_1,y_2)$ in some fiber product of $S$ with itself.
\begin{proof}[Proof of Proposition~\ref{2pt}]
By Cauchy-Schwarz, $|\Lambda_{P_1,P_2}(f_0,f_1,f_2)|^2$ is bounded above by
\[
\|f_0\|^2_{L^2}\E_{x,y_1,y_2}f_1(x+P_1(y_1))f_1(x+P_1(y_2))f_2(x+P_2(y_1))f_2(x+P_2(y_2)).
\]
After the change of variables $x\mapsto x-P_1(y_1)$, the average above becomes
\[
\E_{x,y_1,y_2}f_1(x)f_1(x+P_1(y_2)-P_1(y_1))f_2(x+P_2(y_1)-P_1(y_1))f_2(x+P_2(y_2)-P_1(y_1)).
\]
We rewrite this by collecting together $(y_1,y_2)\in\F_q^2$ in the same fiber of $T_1(y_1,y_2):=P_1(y_2)-P_1(y_1)$:
\[
\f{1}{q^2}\sum_{x,z\in\F_q}f_1(x)f_1(x+z)\f{1}{q}\sum_{\substack{y_1,y_2\in\F_q \\ T_1(y_1,y_2)=z}}f_2(x+P_2(y_1)-P_1(y_1))f_2(x+P_2(y_2)-P_1(y_1)).
\]
Then $|\Lambda_{P_1,P_2}(f_0,f_1,f_2)|^4/(\|f_0\|_{L^2}^4\|f_1\|_{L^2}^4)$ is bounded above by
\[
\f{1}{q^4}\sum_{x,z\in\F_q}\sum_{\substack{y\in\F_q^4 \\ T_1(y_1,y_2)=z \\ T_1(y_3,y_4)=z}}\prod_{i=0}^1f_2(x+P_2(y_{1+2i})-P_1(y_{1+2i}))f_2(x+P_2(y_{2+2i})-P_1(y_{1+2i})),
\]
by Cauchy-Schwarz. Summing the inner sum over $z\in \F_q$, this equals
\begin{equation}\label{2pt1}
\f{1}{q^4}\sum_{\substack{x\in\F_q \\ y\in \F_q^2\times_{T_1}\F_q^2}}\prod_{i=0}^1f_2(x+P_2(y_{1+2i})-P_1(y_{1+2i}))f_2(x+P_2(y_{2+2i})-P_1(y_{1+2i})).
\end{equation}
After making the change of variables $x\mapsto x-P_2(y_1)+P_1(y_1)$, we can rewrite~(\ref{2pt1}) by collecting together $y\in \F_q^2\times_{T_1}\F_q^2$ with the same values of $T_2(y):=P_2'(y_3)-P_2'(y_1)$ and $T_3(y):=P_2(y_2)-P_2(y_1)$:
\begin{equation}\label{2pt2}
\f{1}{q^4}\sum_{x,z,z'\in\F_q }f_2(x)f_2(x+z)f_2(x+z')\sum_{\substack{y\in \F_q^2\times_{T_1}\F_q^2 \\ T_2(y)=z \\ T_3(y)=z'}}f_2(x+z+P_2(y_4)-P_2(y_3)).
\end{equation}

Applying Cauchy-Schwarz to the outer sum in~(\ref{2pt2}) thus shows that
\[
\f{|\Lambda_{P_1,P_2}(f_0,f_1,f_2)|^8}{\|f_0\|_{L^2}^{8}\|f_1\|_{L^2}^{8}\|f_2\|_{L^2}^6}
\]
is bounded above by
\[
\f{1}{q^5}\sum_{x,z,z'\in\F_q}\sum_{\substack{y\in \F_q^2\times_{T_1}\F_q^2 \\ T_2(y_1,y_2,y_3,y_4)=z \\ T_2(y_5,y_6,y_7,y_8)=z \\ T_3(y_1,y_2,y_3,y_4)=z' \\ T_3(y_5,y_6,y_7,y_8)=z'}}f_2(x+z+P_2(y_4)-P_2(y_3))f_2(x+z+P_2(y_8)-P_2(y_7)).
\]
Making the change of variables $x\mapsto x-z-P_2(y_4)+P_2(y_3)$ and summing the inner sum over $(z,z')\in \F_q^2$, the above becomes
\[
\f{|V_{P_1,P_2}(\F_q)|}{q^4}\Lambda'_{P_1,P_2}(f_2,f_2).
\]
\end{proof}

\section{Dimension bounds}\label{3}
The main goal of this section is to prove a power-saving bound for the sum
\begin{equation}\label{charbd}
\sum_{y\in V_{P_1,P_2}(\F_q)}\psi(Q_{P_1,P_2}(y))
\end{equation}
whenever $\psi$ is a non-trivial additive character of $\F_q$. When the characteristic of $\F_q$ is large enough and $Q_{P_1,P_2}$ is not constant on the smooth points of any irreducible component of $V_{P_1,P_2}(\bar\F_q)$, then such a bound should follow from Deligne's theorem and the Grothendieck-Lefschetz trace formula. Indeed, in Proposition 9 of~\cite{Ko}, Kowalski has already carried this argument out in general. Kowalski's proposition is phrased in terms of $q$-Weil numbers, so for the convenience of the reader we state below an immediate consequence of it.
\begin{proposition}[Kowalski, Proposition 9.ii of~\cite{Ko}]\label{charsum}
Let $V\subset\A_{\Z}^n$ be an affine subscheme and $F,G\in\Z[V]$ be regular functions on $V$. Suppose that $\psi$ and $\chi$ are additive and multiplicative characters of $\F_q$, respectively. There exists an $\eta_V>0$ depending only on $V$ and a $c_{V,\deg F,\deg{G}}>0$ depending only on $V$, $\deg F$, and $\deg G$ such that if
\[
|F\1(a)|\leq\eta_V|V(\F_q)|
\]
for every $a\in\F_q$ and the characteristic of $\F_q$ is at least $c_{V,\deg{F},\deg{G}}$, then
\[
\sum_{x\in V(\F_q)}\psi(F(x))\chi(G(x))\ll_{V,\deg{F},\deg{G}} q^{\dim{V(\bar\F_q)}-1/2}
\]
whenever $\psi$ is nontrivial.
\end{proposition}

To check that the hypotheses of Proposition~\ref{charsum} are satisfied, we will show that $\dim V_{P_1,P_2}(\bar\F_q)\leq 4$ and that all of the fibers of $Q_{P_1,P_2}:V_{P_1,P_2}(\bar\F_q)\to\bar{\F}_q$ have dimension at most $3$. This second fact follows from the bound $\dim W_{P_1,P_2}(\bar\F_q)\leq 7$, which we will prove instead of bounding the dimension of the fibers directly. We discuss the reasoning for this prior to the proof of Lemma~\ref{dimbound}.

Our main tool in this section will be the connection between a variety's Hilbert polynomial and its dimension, and we will first briefly review the definitions needed to describe this connection. The following standard material can be found in Chapters 9--11 of~\cite{K} and Chapters 2 and 9 of~\cite{CLO}, for example.

\subsection{Preliminaries}

Let $\F$ be a field and $n\in\N$, and let $\N_0:=\N\cup\{0\}$ denote the nonnegative integers. For any $\alpha=(\alpha_1,\dots,\alpha_n)\in\N_0^n$, we set
\[
y^\alpha:=\prod_{i=1}^n y_i^{\alpha_i}
\]
in the ring $\F[y_1,\dots,y_n]$. The quantity $|\alpha|:=\alpha_1+\dots+\alpha_n$ is the \textit{degree} of the monomial $y^\alpha$. Let $M_n$ denote the set of monomials in the variables $y_1,\dots,y_n$:
\[
M_n:=\{y^\alpha:\alpha\in\N_0^n\}.
\]
Any element $G\in\F[y_1,\dots,y_n]$ may be written as
\begin{equation}\label{poly}
G=\sum_{Y\in M_n}a_Y Y,
\end{equation}
where $a_\alpha=0$ for all but finitely many $\alpha\in\N_0^n$. We say that a monomial $Y\in M_n$ \textit{appears} in $G$ if $a_Y\neq 0$ in the expression~(\ref{poly}). For example, $y_1y_2$ appears in $y_1^2+5y_1y_2$, but $y_1y_2$ does not appear in $y_1^2+y_1y_2-y_1y_2$. The \textit{degree} of $G$ is the maximum degree of all monomials appearing in $G$. For $G'\in\F[y_1,\dots,y_n]$, we will write
\[
G=G'+\text{lower degree terms}
\]
to mean that $G=G'+G''$ for some $G''\in\F[y_1,\dots,y_n]$ such that
\[
\deg{G''}<\min\{\deg{Y}:Y\in M_n\text{ appears in }G'\}.
\]

In order to define the leading term of a multivariate polynomial, we must specify an ordering of the set of monomials $M_n$. The most useful orders on $M_n$ are \textit{monomial orders}, which are those that respect multiplication of monomials.
\begin{definition}[Monomial order]
A \emph{monomial order} $>$ on $\F[y_1,\dots,y_n]$ is a total order on $M_n$ that satisfies
\begin{enumerate}
\item $y^\alpha>1$ for all $y^\alpha\in M_n\sm\{1\}$, and
\item if $y^\alpha>y^\beta$, then $y^\alpha\cdot y^\gamma>y^\beta\cdot y^\gamma$ whenever $y^\alpha,y^\beta,y^\gamma\in M_n$.
\end{enumerate}
A monomial order $>$ is \emph{graded} if $y^\alpha>y^\beta$ whenever $|\alpha|>|\beta|$.
\end{definition}
It is an easy consequence of the Hilbert basis theorem that any monomial ordering is a well-ordering.

Once we have specified a monomial order on $\F[y_1,\dots,y_n]$, we may write any $G\in\F[y_1,\dots,y_n]$ as
\[
G=a_{\alpha^1}y^{\alpha^1}+\dots+a_{\alpha^m}y^{\alpha^m},
\]
where $y^{\alpha^1}>\dots>y^{\alpha^m}$ and $a_{\alpha^1}\neq 0$. Then the \textit{leading term} of $G$ is
\[
\lt(G):= a_{\alpha^1}y^{\alpha^1}
\]
and the \textit{leading monomial} of $G$ is
\[
\lm(G):= y^{\alpha^1}.
\]
An important concept for us will be the ideal of leading terms of a set:
\begin{definition}[Leading term ideal]
Let $\F$ be a field and $n\in\N$, and fix a monomial order on $\F[y_1,\dots,y_n]$. For any $S\subset\F[y_1,\dots,y_n]$, the \emph{leading term ideal} of $S$ is
\[
\lt(S):=\langle\{\lt(H):H\in S\}\rangle.
\]
\end{definition}

Now, for any $S\subset\F[y_1,\dots,y_n]$, set
\[
V(S):=\{y\in\F^n:H(y)=0\text{ for all }H\in S\}.
\]
For any $s\in\N_0$, let $\F[y_1,\dots,y_n]_{\leq s}$ denote the $\F$-vector space of polynomials in $\F[y_1,\dots,y_n]$ of degree at most $s$, and for any ideal $I\subset\F[y_1,\dots,y_n]$, set $I_{\leq s}:=I\cap\F[y_1,\dots,y_n]_{\leq s}$. When $s$ is sufficiently large depending on $I$, the \textit{affine Hilbert function} of $I$,
\[
\HF_I(s):=\dim_\F\F[y_1,\dots,y_n]_{\leq s}/I_{\leq s}
\]
equals a polynomial $\HP_I$ called the \textit{affine Hilbert polynomial} of $I$.

Fix a graded monomial order on $\F[y_1,\dots,y_n]$. Then $\HF_I=\HF_{\lt(I)}$, and when $\F$ is algebraically closed, we have that $\deg\HP_I=\dim V(I)$ as well. Thus, 
\[
\dim V(I)=\deg\HP_{\lt(I)}
\]
when $\F$ is algebraically closed. The degree of $\HP_{\lt(I)}$ is easy to compute if one knows a generating set for $\lt(I)$. As one important special case, if there exist $G_1,\dots,G_n\in I$ and $\alpha_1,\dots,\alpha_n>0$ such that
\[
\lm(G_i)=y_i^{\alpha_i}
\]
for each $i=1,\dots,n$, then $\dim V(I)=0$.

Our choice of graded monomial order will not have much of an impact on our arguments bounding $\dim V_{P_1,P_2}(\bar\F_q)$ and $\dim W_{P_1,P_2}(\bar\F_q)$. For this reason, we will use the graded lexicographic order, which is simple to describe.
\begin{definition}
Let $\F$ be a field and $n\in\N$. The \emph{graded lexicographic order} (abbreviated \emph{grlex}) with $y_1>\dots>y_n$ is defined as follows. We have $y^\alpha>y^\beta$ if
\begin{enumerate}
\item $|\alpha|>|\beta|$, or
\item $|\alpha|=|\beta|$ and $\alpha_{i_0}>\beta_{i_0}$, where $i_0\in[n]$ is the smallest index $i$ for which $\alpha_i\neq\beta_i$.
\end{enumerate}
\end{definition}
For example, $y_1^2y_2y_3>y_1^2y_3^2$ with respect to the grlex ordering with $y_1>y_2>y_3$.

\subsection{Bounding the dimension of $V_{P_1,P_2}(\bar\F_q)$ and $W_{P_1,P_2}(\bar\F_q)$}

It is now immediate that $\dim{V_{P_1,P_2}(\bar\F_q)}\leq 4$ when $\Char\F_q\gg_{P_1,P_2}1$. Indeed, put the grlex order with
\[
y_8>y_4>y_7>y_3>y_6>y_2>y_5>y_1
\]
on $\bar\F_q[y_1,\dots,y_8]$, and let
\[
I=\langle R_{P_1,P_2}^{(1)},R_{P_1,P_2}^{(2)},R_{P_1,P_2}^{(3)},R_{P_1,P_2}^{(4)}\rangle\subset\bar\F_q[y_1,\dots,y_8],
\]
so that $V(I)=V_{P_1,P_2}(\bar\F_q)$. If $\Char\F_q>\max(|a_{r_1}|,|b_{r_2}|,|c_{r_2}|)$, then
\[
\lm(R_{P_1,P_2}^{(1)})=y_4^{r_1},
\]
\[
\lm(R_{P_1,P_2}^{(2)})=y_8^{r_1},
\]
\[
\lm(R_{P_1,P_2}^{(3)})=y_6^{r_2},
\]
and
\[
\lm(R_{P_1,P_2}^{(4)})=y_7^{r_2}.
\]
Thus, $\lt(I)\supset\langle y_4^{r_1},y_6^{r_2},y_7^{r_2},y_8^{r_1}\rangle$, so that $\dim V_{P_1,P_2}(\bar\F_q)\leq 4$. (In fact, the $R_{P_1,P_2}^{(i)}$'s form a Gr\"obner basis, so actually $\lt(I)=\langle y_4^{r_1},y_6^{r_2},y_7^{r_2},y_8^{r_1}\rangle$.)

That $|V_{P_1,P_2}(\F_q)|\ll_{P_1,P_2}q^4$ is now a consequence of the following corollary of the Lang-Weil bound:
\begin{theorem}[Lang and Weil, Lemma 1 of~\cite{LW}]
Let $m,n,r\in\N$ and $\F$ be a finite field. Suppose that $V=V(\{G_1,\dots,G_m\})\subset\A_{\bar\F}^n$ is an affine variety with $\deg{G_i}\leq r$ for every $i=1,\dots,m$. Then
\[
|V\cap\F|\ll_{m,n,r}|\F|^{\dim V}.
\]
\end{theorem}
As a corollary of Proposition~\ref{2pt}, we can thus bound $\Lambda_{P_1,P_2}(f_0,f_1,f_2)$ in terms of $\Lambda_{P_1,P_2}'(f_2,f_2)$:
\begin{corollary}\label{cor}
Suppose that $f_0,f_1,f_2:\F_q\to\R$. Then
\[
\Lambda_{P_1,P_2}(f_0,f_1,f_2)\ll_{P_1,P_2}\|f_0\|_{L^2}\|f_1\|_{L^2}\|f_2\|_{L^2}^{3/4}|\Lambda_{P_1,P_2}'(f_2,f_2)|^{1/8}.
\]
\end{corollary}

Proving that $\dim W_{P_1,P_2}(\bar\F_q)\leq 7$, however, is not as simple. Regardless of which graded monomial order we put on $\F[y_1,\dots,y_{16}]$, two of the defining polynomials of $W_{P_1,P_2}(\F)$ will have leading monomial equal to a power of the same $y_i$. Indeed, let
\begin{align*}
I =\langle& R_{P_1,P_2}^{(1)}(y_1,\dots,y_8),R_{P_1,P_2}^{(1)}(y_9,\dots,y_{16}), \\
& R_{P_1,P_2}^{(2)}(y_1,\dots,y_8),R_{P_1,P_2}^{(2)}(y_9,\dots,y_{16}), \\
& R_{P_1,P_2}^{(3)}(y_1,\dots,y_8),R_{P_1,P_2}^{(3)}(y_9,\dots,y_{16}), \\
& R_{P_1,P_2}^{(4)}(y_1,\dots,y_8),R_{P_1,P_2}^{(4)}(y_9,\dots,y_{16}), \\
&Q_{P_1,P_2}(y_1,\dots,y_8)-Q_{P_1,P_2}(y_9,\dots,y_{16})\rangle.
\end{align*}
Then, for every $i=1,\dots,16$, some power of the variable $y_i$ appears in at least two of the generators of $I$ as a monomial of the highest degree.

Because working in a sixteen-variable polynomial ring has the potential to become very messy, we will simplify things by intersecting $W_{P_1,P_2}$ with seven well-chosen hyperplanes. We will then show that the resulting variety is zero-dimensional. Here we need the following result on intersections of varieties, which can be found as Proposition I.7.1 in~\cite{H}.
\begin{proposition}\label{int}
Let $n\in\N$ and $\F$ be any field. Suppose that $W_1,W_2\subset\A_\F^n$ are two irreducible affine varieties, and that $Z\subset W_1\cap W_2$ is an irreducible component of $W_1\cap W_2$. Then
\[
\dim W_1+\dim W_2-n\leq\dim Z.
\]
\end{proposition}

Before proving Lemma~\ref{dimbound}, we remark on why we bound $\dim W_{P_1,P_2}(\bar\F_q)$ instead of bounding the fibers of $Q_{P_1,P_2}$ directly. The reason is that it turns out to be very convenient to work over $\bar{\Q}$, so we want all of our defining polynomials to have coefficients in $\Q$. The proof of Lemma~\ref{dimbound} reduces the problem of bounding $\dim W_{P_1,P_2}(\bar\Q)$ to two inequalities that involve explicit elements of $\bar\Q$ and are straightforward to prove. Once we have shown that $\dim W_{P_1,P_2}(\bar\Q)\leq 7$, we can deduce the same bound for $\dim W_{P_1,P_2}(\bar\F_q)$ whenever the characteristic of $\F_q$ is sufficiently large.

Indeed, Exercise II.3.20 of~\cite{H} tells us that $\dim W_{P_1,P_2}(\Q)=\dim W_{P_1,P_2}(\bar\Q)$. Now set $d=\dim W_{P_1,P_2}(\Q)$ and let $I$ be as above, viewed as an ideal in $\Q[y_1,\dots,y_{16}]$. Then Noether normalization says that there exist (algebraically independent) elements $x_1,\dots,x_d\in\Q[y_1,\dots,y_{16}]/I$ such that $\Q[y_1,\dots,y_{16}]/I$ is integral over $\Q[x_1,\dots,x_d]$. Thus, each $y_i$, $i=1,\dots,16$, satisfies a monic polynomial equation
\[
y_i^{s_i}+v_{i,s_i-1}y_i^{s_i-1}+\dots+v_{i,0}
\]
with $v_{i,j}\in\Q[x_1,\dots,x_d]$. Let $S\subset\Q$ be the (finite) set of all coefficients of terms (i.e. products of the $y_1,\dots,y_{16}$'s) appearing in the $x_k$'s and coefficients of terms (i.e. products of the $x_1,\dots,x_d$'s) appearing in the $v_{i,j}$'s, and let $M$ be some fixed common multiple of the denominators of the elements of $S$. Then $\Z[1/M][y_1,\dots,y_{16}]/I$ (with $I$ as above, but as an ideal in $\Z[1/M]$) is a finitely-generated $\Z[1/M][x_1,\dots,x_d]$-module, and tensoring with $\bar\F_q$ of characteristic at least $M$, we see that $\dim{W_{P_1,P_2}(\bar\F_q)}\leq d$. With a little more work, one can show that $\dim W_{P_1,P_2}(\bar\F_q)=\dim W_{P_1,P_2}(\Q)$ when $\F_q$ has large enough characteristic, but we will only need the upper bound.

\begin{lemma}\label{dimbound}
There exists a $c_{P_1,P_2}>0$ depending only on $P_1$ and $P_2$ such that
\[
\dim W_{P_1,P_2}(\bar\F_q)\leq 7
\]
whenever the characteristic of $\F_q$ is at least $c_{P_1,P_2}$.
\end{lemma}
\begin{proof}
By the discussion above, it suffices to show that $\dim W_{P_1,P_2}(\bar\Q)\leq 7$. Let $W_1$ be a top-dimensional irreducible component of $W_{P_1,P_2}(\bar\Q)$, and suppose that $W_2\subset\A_{\bar\Q}^{16}$ is an irreducible affine variety of dimension $9$ such that $W_1\cap W_2\neq\emptyset$. Then Proposition~\ref{int} implies that
\[
\dim W_{P_1,P_2}(\bar\Q)\leq 7+\dim(W_1\cap W_2)\leq 7+\dim(W_{P_1,P_2}(\bar\Q)\cap W_2).
\]
It thus suffices to find such a $W_2$ for which $\dim(W_{P_1,P_2}(\bar\Q)\cap W_2)=0$.

First suppose that $r_1<r_2$. Write $r_i=(r_1,r_2)r_i'$ for $i=1,2$, so that $(r_1',r_2')=1$. Let $r$ be any integer whose reduction modulo $r_2'$ is the multiplicative inverse of $r_1'$ modulo $r_2'$, and fix some $w=(w_1,\dots,w_{16})\in W_1$. Set $u_1=w_8-e_{r_2}(r)w_7,u_2=w_3,u_3=w_5,u_4=w_{16}-e_{r_2}(r)w_{15},u_5=w_{11},u_6=w_{13},$ and $u_7=w_{12}$. Then we take
\begin{align*}
W_{2}=V(\langle &y_8-e_{r_2}(r)y_7-u_1,y_3-u_2,y_5-u_3, \\
&y_{16}-e_{r_2}(r)y_{15}-u_4,y_{11}-u_5,y_{13}-u_6,y_{12}-u_7\rangle),
\end{align*}
which is a $9$-dimensional irreducible subvariety of $\A_{\bar\Q}^{16}$ such that $W_1\cap W_2\neq\emptyset$. Also set
\begin{align*}
I =\langle& R_{P_1,P_2}^{(1)}(y_1,\dots,y_8),R_{P_1,P_2}^{(1)}(y_9,\dots,y_{16}), \\
& R_{P_1,P_2}^{(2)}(y_1,\dots,y_8),R_{P_1,P_2}^{(2)}(y_9,\dots,y_{16}), \\
& R_{P_1,P_2}^{(3)}(y_1,\dots,y_8),R_{P_1,P_2}^{(3)}(y_9,\dots,y_{16}), \\
& R_{P_1,P_2}^{(4)}(y_1,\dots,y_8),R_{P_1,P_2}^{(4)}(y_9,\dots,y_{16}), \\
&Q_{P_1,P_2}(y_1,\dots,y_8)-Q_{P_1,P_2}(y_9,\dots,y_{16}), \\
&y_8-e_{r_2}(r)y_7-u_1,y_3-u_2,y_5-u_3, \\
&y_{16}-e_{r_2}(r)y_{15}-u_4,y_{11}-u_5,y_{13}-u_6,y_{12}-u_7\rangle,
\end{align*}
so that $W_{P_1,P_2}(\bar\Q)\cap W_2=V(I)$.

We put the grlex order with
\[
y_8>y_4>y_7>y_3>y_6>y_2>y_5>y_1,
\]
\[
y_{16}>y_{12}>y_{15}>y_{11}>y_{14}>y_{10}>y_{13}>y_9,
\]
and $y_1>y_{16}$ on $\bar\Q[y_1,\dots,y_{16}]$. By reducing the generating polynomials of $W_{P_1,P_2}(\bar\Q)$ modulo the generating polynomials of $W_2$ and dividing by either $a_{r_1}$ or $b_{r_2}$ (in this case, $c_{r_2}=b_{r_2}$), we see that $I$ contains polynomials of the form
\[
y_4^{r_1}-y_2^{r_1}+y_1^{r_1}+\text{lower degree terms},
\]
\[
(e_{r_2'}(1)-1)y_7^{r_1}-y_6^{r_1}+\text{lower degree terms},
\]
\[
y_6^{r_2}-y_2^{r_2}+y_1^{r_2}+\text{lower degree terms},
\]
\[
y_7^{r_2}+y_1^{r_2}+\text{lower degree terms},
\]
\[
-y_{10}^{r_1}+y_9^{r_1}+\text{lower degree terms},
\]
\[
(e_{r_2'}(1)-1)y_{15}^{r_1}-y_{14}^{r_1}+\text{lower degree terms},
\]
\[
y_{14}^{r_2}-y_{10}^{r_2}+y_9^{r_2}+\text{lower degree terms},
\]
\[
y_{15}^{r_2}+y_9^{r_2}+\text{lower degree terms},
\]
and
\[
-y_4^{r_2}+\text{lower degree terms},
\]
in addition to the polynomials $y_8-e_{r_2}(r)y_7-u_1,y_3-u_2,y_5-u_3,y_{16}-e_{r_2}(r)y_{15}-u_4,y_{11}-u_5,y_{13}-u_6,$ and $y_{12}-u_7$. Thus, it is immediate that $I$ contains polynomials with leading terms equal to
\[
y_3,y_4^{r_1},y_5,y_6^{r_2},y_7^{r_1},y_8,y_{10}^{r_1},y_{11},y_{12},y_{13},y_{14}^{r_2},y_{15}^{r_1},\text{ and }y_{16}.
\]
To show that $\dim V(I)=0$, then, it remains to show that there exist $G_1,G_2,G_9\in I$ and $\alpha_1,\alpha_2,\alpha_9>0$ such that $\lm(G_i)=y_i^{\alpha_i}$ for each $i=1,2,9$.

Note that $I$ also contains elements of the form
\[
y_4^{[r_1,r_2]}-(y_2^{r_1}-y_1^{r_1})^{r_2'}+\text{lower degree terms},
\]
\[
(e_{r_2'}(1)-1)^{r_2'}y_7^{[r_1,r_2]}-y_6^{[r_1,r_2]}+\text{lower degree terms},
\]
\[
y_6^{[r_1,r_2]}-(y_2^{r_2}-y_1^{r_2})^{r_1'}+\text{lower degree terms},
\]
\[
y_7^{[r_1,r_2]}-(-y_1^{r_2})^{r_1'}+\text{lower degree terms},
\]
and
\[
y_4^{[r_1,r_2]}+\text{lower degree terms}.
\]
Thus, there exist $H_1,H_2\in I$ with
\begin{align*}
H_1&=y_4^{[r_1,r_2]}-(y_4^{[r_1,r_2]}-(y_2^{r_1}-y_1^{r_1})^{r_2'})+\text{lower degree terms} \\
&=(y_2^{r_1}-y_1^{r_1})^{r_2'}+\text{lower degree terms}
\end{align*}
and
\begin{align*}
H_2=&(e_{r_2'}(1)-1)^{r_2'}(y_7^{[r_1,r_2]}-(-y_1^{r_2})^{r_1'})-((e_{r_2'}(1)-1)^{r_2'}y_7^{[r_1,r_2]}-y_6^{[r_1,r_2]}) \\
&-(y_6^{[r_1,r_2]}-(y_2^{r_2}-y_1^{r_2})^{r_1'})+\text{lower degree terms} \\
=& (y_2^{r_2}-y_1^{r_2})^{r_1'}-(e_{r_2'}(1)-1)^{r_2'}(-y_1^{r_2})^{r_1'}+\text{lower degree terms}.
\end{align*}
Note that $\lm(H_1)=\lm(H_2)=y_2^{[r_1,r_2]}$. Now set
\[
H_1'=(y_2^{r_1}-y_1^{r_1})^{r_2'}
\]
and
\[
H_2'=(y_2^{r_2}-y_1^{r_2})^{r_1'}-(e_{r_2'}(1)-1)^{r_2'}(-y_1^{r_2})^{r_1'},
\]
so that $H_1'$ and $H_2'$ are both homogeneous polynomials of degree $[r_1,r_2]$ in $\bar\Q[y_1,y_2]$.

Let
\begin{align*}
\mathcal{H}=&\{(y_1^iy_2^{[r_1,r_2]-1-i})H_1':0\leq i\leq [r_1,r_2]-1\} \\ &\cup\{(y_1^iy_2^{[r_1,r_2]-1-i})H_2':0\leq i\leq [r_1,r_2]-1\}.
\end{align*}
If the polynomials in $\HH$ generate the $\bar\Q$-vector space of homogeneous degree $2[r_1,r_2]-1$ polynomials in $\bar\Q[y_1,y_2]$, then $\langle H_1,H_2\rangle$, and thus $I$, certainly contains an element with leading monomial $y_1^{2[r_1,r_2]-1}$. Since $|\mathcal{H}|=2[r_1,r_2]$ and the space of homogeneous degree $2[r_1,r_2]-1$ polynomials in $\Q[y_1,y_2]$ has dimension $2[r_1,r_2]$, to show that $I$ contains an element with leading monomial $y_1^{2[r_1,r_2]-1}$ it now suffices to show that the polynomials in $\mathcal{H}$ are linearly independent.

If the elements of $\HH$ were not linearly independent, then there would exist nonzero homogeneous polynomials $F_1,F_2\in\bar\Q[y_1,y_2]$ of degree $[r_1,r_2]-1$ such that $F_1H_2'=F_2H_1'$. As a consequence, we certainly have that $F_1(1,y_2)H_2'(1,y_2)=F_2(1,y_2)H_1'(1,y_2)$ in $\bar\Q[y_2]$, and since $H_1'$ and $H_2'$ both have degree $[r_1,r_2]$ in $y_2$, this implies that the polynomials
\[
H''_1=H_1'(1,z)=(z^{r_1}-1)^{r_2'}
\]
and 
\[
H''_2=H_2'(1,z)=(z^{r_2}-1)^{r_1'}-(-1)^{r_1'}(e_{r_2'}(1)-1)^{r_2'}
\]
must have a common root over $\bar\Q$. We will show that this is impossible when $r_1<r_2$.

Let $\omega$ be a root of $H''_1$, so $\omega=e_{r_1}(a)$ for some $a\in\Z$. Then
\[
H''_2(\omega)=(e_{r_1'}(ar_2')-1)^{r_1'}-(-1)^{r_1'}(e_{r_2'}(1)-1)^{r_2'}.
\]
Since $r_2'>r_1'$, we have $|e_{r_2'}(1)-1|^{r_2'}<|e_{r_1'}(ar_2')-1|^{r_1'}$ for every $a$ for which $e_{r_1'}(ar_2')\neq 1$, and when $e_{r_1'}(ar_2')=1$, we have $|e_{r_1'}(ar_2')-1|^{r_1'}=0<|e_{r_2'}(1)-1|^{r_2'}$. So, $H''_2(\omega)\neq 0$.

Thus, $I$ contains a polynomial $G_1$ with $\lm(G_1)=y_1^{2[r_1,r_2]-1}$ and a polynomial $G_2=H_1$ with $\lm(G_2)=y_2^{[r_1,r_2]}$. Since $I$ also contains $y_{12}^{r_2}-u_7^{r_2}$ and a polynomial of the form
\[
y_{12}^{r_1}-y_{10}^{r_1}+y_9^{r_1}+\text{lower degree terms},
\]
the proof that $I$ contains an element with leading monomial $y_9^{2[r_1,r_2]-1}$ is identical to the argument just given, but with $y_1,y_2,y_4,y_6,$ and $y_7$ replaced by $y_9,y_{10},y_{12},y_{14},$ and $y_{15}$, respectively. We conclude that $\dim V(I)=0$.

Now suppose that $r_1=r_2$. In this case, we have that
\begin{equation}\label{qprime}
Q_{P_1,P_2}':=Q_{P_1,P_2}+\f{b_{r_2}}{a_{r_1}}(R_{P_1,P_2}^{(1)}-R_{P_1,P_2}^{(2)})+R_{P_1,P_2}^{(3)}
\end{equation}
equals
\[
P_3(y_1)-P_3(y_2)-P_3(y_3)+P_3(y_4)-P_3(y_5)+P_3(y_6)+P_3(y_7)-P_3(y_8).
\]
As before, fix some $w\in (w_1,\dots,w_{16})\in W_1$ and set $u_1=w_8-w_3,u_2=w_4-w_5,u_3=w_1,u_4=w_{16}-w_{11},u_5=w_{12}-w_{13},u_6=w_9,$ and $u_7=w_{15}^{r_3}-2w_{11}^{r_3}+w_{14}^{r_3}-w_{10}^{r_3}$. We take
\begin{align*}
W_2= V(\langle &y_8-y_3-u_1,y_4-y_5-u_2,y_1-u_3, \\
&y_{16}-y_{11}-u_4,y_{12}-y_{13}-u_5,y_9-u_6,y_{15}^{r_3}-2y_{11}^{r_3}+y_{14}^{r_3}-y_{10}^{r_3}-u_7\rangle),
\end{align*}
which is a $9$-dimensional irreducible (apply Eisenstein's criterion to $y_{15}^{r_3}-2y_{11}^{r_3}+y_{14}^{r_3}-y_{10}^{r_3}-u_7$ in $\bar\Q[y_{10},y_{14},y_{15}][y_{11}]$) subvariety of $\A_{\bar\Q}^{16}$ such that $W_1\cap W_2\neq\emptyset$. Set
\begin{align*}
I =\langle& R_{P_1,P_2}^{(1)}(y_1,\dots,y_8),R_{P_1,P_2}^{(1)}(y_9,\dots,y_{16}), \\
& R_{P_1,P_2}^{(2)}(y_1,\dots,y_8),R_{P_1,P_2}^{(2)}(y_9,\dots,y_{16}), \\
& R_{P_1,P_2}^{(3)}(y_1,\dots,y_8),R_{P_1,P_2}^{(3)}(y_9,\dots,y_{16}), \\
& R_{P_1,P_2}^{(4)}(y_1,\dots,y_8),R_{P_1,P_2}^{(4)}(y_9,\dots,y_{16}), \\
&Q_{P_1,P_2}'(y_1,\dots,y_8)-Q_{P_1,P_2}'(y_9,\dots,y_{16}) \\
&y_8-y_3-u_1,y_4-y_5-u_2,y_1-u_3, \\
&y_{16}-y_{11}-u_4,y_{12}-y_{13}-u_5,y_9-u_6,y_{15}^{r_3}-2y_{11}^{r_3}+y_{14}^{r_3}-y_{10}^{r_3}-u_7\rangle,
\end{align*}
which also contains $Q_{P_1,P_2}(y_1,\dots,y_8)-Q_{P_1,P_2}(y_9,\dots,y_{16})$ by~(\ref{qprime}). As above, we have $W_{P_1,P_2}(\bar\Q)\cap W_2=V(I)$.

We put the grlex order with
\[
y_8>y_4>y_7>y_3>y_6>y_5>y_2>y_1,
\]
\[
y_{16}>y_{12}>y_{15}>y_{11}>y_{14}>y_{13}>y_{10}>y_9,
\]
and $y_1>y_{16}$ on $\bar\Q[y_1,\dots,y_{16}]$. This is almost the same as the order used in the previous case, except that we have swapped $y_2$ with $y_5$ and $y_{10}$ with $y_{13}$. Reducing the generating polynomials of $W_{P_1,P_2}(\bar\Q)$ with $Q_{P_1,P_2}(y_1,\dots,y_8)-Q_{P_1,P_2}(y_9,\dots,y_{16})$ replaced by $Q_{P_1,P_2}'(y_1,\dots,y_8)-Q_{P_1,P_2}'(y_9,\dots,y_{16})$ by the generating polynomials of $W_2$ and dividing by either $a_{r_1},b_{r_1},c_{r_1},$ or $d_{r_3}$, we get that $I$ contains polynomials of the form
\[
-y_3^{r_1}+y_5^{r_1}-y_2^{r_1}+\text{lower degree terms},
\]
\[
-y_7^{r_1}+y_3^{r_1}-y_6^{r_1}+y_5^{r_1}+\text{lower degree terms},
\]
\[
y_6^{r_1}-y_5^{r_1}-y_2^{r_1}+\text{lower degree terms},
\]
\[
y_7^{r_1}-y_3^{r_1}-y_5^{r_1}+\text{lower degree terms},
\]
\[
-y_{11}^{r_1}+y_{13}^{r_1}-y_{10}^{r_1}+\text{lower degree terms},
\]
\[
-y_{15}^{r_1}+y_{11}^{r_1}-y_{14}^{r_1}+y_{13}^{r_1}+\text{lower degree terms},
\]
\[
y_{14}^{r_1}-y_{13}^{r_1}-y_{10}^{r_1}+\text{lower degree terms},
\]
\[
y_{15}^{r_1}-y_{11}^{r_1}-y_{13}^{r_1}+\text{lower degree terms},
\]
and
\begin{equation}\label{seven1}
-y_7^{r_3}+2y_3^{r_3}-y_6^{r_3}+y_2^{r_3}+\text{lower degree terms},
\end{equation}
in addition to the polynomials $y_8-y_3-u_1,y_4-y_5-u_2,y_1-u_3,y_{16}-y_{11}-u_4,y_{12}-y_{13}-u_5,y_9-u_6,$ and $y_{15}^{r_3}-2y_{11}^{r_3}+y_{14}^{r_3}-y_{10}^{r_3}-u_7$. Thus, the ideal $I$ contains polynomials with leading terms equal to
\[
y_1,y_3^{r_1},y_4,y_6^{r_1},y_7^{r_1},y_8,y_9,y_{11}^{r_1},y_{12},y_{14}^{r_1},y_{15}^{r_1},\text{ and }y_{16},
\]
and to prove that $\dim V(I)=0$, it suffices to show that there exist polynomials $G_2,G_5,G_{10},G_{13}\in I$ and $\alpha_2,\alpha_5,\alpha_{10},\alpha_{13}>0$ such that $\lm(G_i)=y_i^{\alpha_i}$ for each $i=2,5,10,13$.

Note that $I$ also contains elements of the form
\begin{equation}\label{six}
y_6^{r_1}+\text{lower degree terms},
\end{equation}
\begin{equation}\label{seven2}
y_7^{r_1}-y_3^{r_1}+y_2^{r_1}+\text{lower degree terms},
\end{equation}
\begin{equation}\label{three}
y_3^{r_1}+2y_2^{r_1}+\text{lower degree terms},
\end{equation}
and
\[
y_5^{r_1}+y_2^{r_1}+\text{lower degree terms},
\]
so that $y_5^{r_1}\in\lt(I)$.

Since $I$ contains polynomials of the form~(\ref{seven1}) and~(\ref{seven2}), it certainly contains ones of the form
\[
\left(\sum_{i=0}^{r_1-1}(y_7^{r_3})^i(2y_3^{r_3}+y_2^{r_3})^{r_1-1-i}\right)y_6^{r_3}+y_7^{r_1r_3}-(2y_3^{r_3}+y_2^{r_3})^{r_1}+\text{lower degree terms}
\]
and
\[
y_7^{r_1r_3}-(y_3^{r_1}-y_2^{r_1})^{r_3}+\text{lower degree terms},
\]
and hence of the form
\begin{align*}
&\left(\sum_{i=0}^{r_1-1}(y_7^{r_3})^i(2y_3^{r_3}+y_2^{r_3})^{r_1-1-i}\right)y_6^{r_3}-((2y_3^{r_3}+y_2^{r_3})^{r_1}-(y_3^{r_1}-y_2^{r_1})^{r_3}) \\
&+\text{lower degree terms}
\end{align*}
as well. As $I$ also contains a polynomial of the form
\begin{align*}
&\left(\sum_{i=0}^{r_1-1}(y_7^{r_3})^i(2y_3^{r_3}+y_2^{r_3})^{r_1-1-i}\right)^{r_1}y_6^{r_1r_3}-((2y_3^{r_3}+y_2^{r_3})^{r_1}-(y_3^{r_1}-y_2^{r_1})^{r_3})^{r_1} \\
&+\text{lower degree terms}
\end{align*}
and, since $I$ contains~(\ref{six}), a polynomial of the form
\[
\left(\sum_{i=0}^{r_1-1}(y_7^{r_3})^i(2y_3^{r_3}+y_2^{r_3})^{r_1-1-i}\right)^{r_1}y_6^{r_1r_3}+\text{lower degree terms},
\]
we see that $I$ contains an element $H_1$ of the form
\[
H_1=((2y_3^{r_3}+y_2^{r_3})^{r_1}-(y_3^{r_1}-y_2^{r_1})^{r_3})^{r_1}+\text{lower degree terms}.
\]
In addition, $I$ contains a polynomial of the form~(\ref{three}), so it also contains a polynomial $H_2$ of the form
\[
H_2=(y_3^{r_1}+2y_2^{r_1})^{r_1r_3}+\text{lower degree terms}.
\]

Now set
\[
H_1'(y_2,y_3)=(y_3^{r_1}+2y_2^{r_1})^{r_1r_3}
\]
and
\[
H_2'(y_2,y_3)=((2y_3^{r_3}+y_2^{r_3})^{r_1}-(y_3^{r_1}-y_2^{r_1})^{r_3})^{r_1},
\]
which are both homogeneous polynomials of degree $r_1^2r_3$ in $\bar\Q[y_2,y_3]$. By a similar argument as the one given in the previous case, it follows that $I$ contains a polynomial with leading monomial equal to $y_2^{2r_1^2r_3}$ if there do not exist nonzero homogeneous polynomials $F_1,F_2\in\bar\Q[y_2,y_3]$ of degree $r_1^2r_3-1$ such that $F_1H_2'=F_2H_1'$. Because $H_1'$ and $H_2'$ both have degree $r_1^2r_3$ in $y_3$, if there did exist such polynomials, then
\[
H_1''=z^{r_1}+2
\]
and
\[
H_2''=(2z^{r_3}+1)^{r_1}-(z^{r_1}-1)^{r_3}
\]
would share a common root. We will show that this is impossible when $r_1<r_3$.

Suppose that $\omega\in\bar\Q$ is a root of $H_1''$. Then $\omega=e_{r_1}(a/2)2^{1/r_1}$ for some odd $a\in\Z$, and we have
\[
H_2''(\omega)=(2^{r_3/r_1+1}e_{r_1}(ar_3/2)+1)^{r_1}-(-3)^{r_3}.
\]
Note that
\[
|(2^{r_3/r_1+1}e_{r_1}(ar_3/2)+1)^{r_1}|\geq(2^{r_3/r_1+1}-1)^{r_1}
\]
for all $a\in\Z$, so for $\omega$ to be a root of $H_2''$, we need that
\[
3^{r_3/r_1}\geq 2^{r_3/r_1+1}-1.
\]
However, as can be easily checked, the function $x\mapsto 2^{x+1}-3^x-1$ is always positive on the interval $(0,1)$. Since $0<r_3/r_1<1$, we see that $H_2''(\omega)\neq 0$.

Thus, $I$ contains elements with leading monomial equal to $y_2^{r_1^2r_3}$ and $y_5^{r_1}$. As before, the proof that $I$ contains elements with leading monomials equal to $y_{10}^{r_1^2r_3}$ and a power of $y_{13}^{r_1}$ is the same. We conclude that $\dim V(I)=0$ in this case as well.
\end{proof}

\section{Proof of Theorem~\ref{main}}\label{4}

We can now deduce from Proposition~\ref{charsum} and Lemma~\ref{dimbound} the character sum bound needed to complete the proof of Theorem~\ref{main}.
\begin{proposition}\label{last}
There exists a $c_{P_1,P_2}>0$ depending only on $P_1$ and $P_2$ such that if the characteristic of $\F_q$ is at least $c_{P_1,P_2}$ and $\psi\in\widehat{\F}_q$ is nontrivial, then
\[
\E_{y\in V_{P_1,P_2}(\F_q)}\psi(Q_{P_1,P_2}(y))\ll_{P_1,P_2}q^{-1/2}.
\]
\end{proposition}
\begin{proof}
That $V_{P_1,P_2}(\bar\F_q)$ and $Q_{P_1,P_2}$ satisfy the hypotheses of Proposition~\ref{charsum} follows from Lemma~\ref{dimbound} and the Lang-Weil bound, which together tell us that
\[
\sum_{a\in\F_q}|Q\1_{P_1,P_2}(a)|^2=|W_{P_1,P_2}(\F_q)|\ll_{P_1,P_2}q^7,
\]
so that every fiber $Q_{P_1,P_2}\1(a)$ of $Q_{P_1,P_2}$ must have dimension at most $3$ when $q$ and the characteristic of $\F_q$ are sufficiently large. Hence, $|Q_{P_1,P_2}\1(a)|\ll_{P_1,P_2}q^3$. By the argument given in Section~\ref{3}, we know that $|V_{P_1,P_2}(\F_q)|\ll_{P_1,P_2}q^4$, so it only remains to check that $|V_{P_1,P_2}(\F_q)|\gg_{P_1,P_2}q^4$. This will imply that $|Q_{P_1,P_2}\1(a)|\ll_{P_1,P_2}|V_{P_1,P_2}(\F_q)|/q$ for all $a\in\F_q$, so that our desired bound will hold when $q$ is sufficiently large.

That $|V_{P_1,P_2}(\F_q)|\geq q^4$ follows easily from two applications of Cauchy-Schwarz. Indeed, if $S$ and $S'$ are any two finite sets and $f:S\to S'$, then
\begin{equation}\label{size}
|S|^2=\left(\sum_{s'\in S'}|f\1(s')|\right)^2\leq |f(S)|\sum_{s'\in S'}|f\1(s')|^2,
\end{equation}
Applying~(\ref{size}) with the function $T_1:\F_q^2\to\F_q$ defined, as in the proof of Proposition~\ref{2pt}, by
\[
T_1(y_1,y_2)=P_1(y_2)-P_1(y_1)
\]
gives the bound $|\F_q^2\times_{T_1}\F_q^2|\geq q^3$, and then applying~(\ref{size}) again with the function from $\F_q^2\times_{T_1}\F_q^2$ to $\F_q^2$ defined by
\[
(y_1,\dots,y_4)\mapsto\begin{pmatrix}
P_2(y_2)-P_2(y_1) \\
P_2'(y_3)-P_2'(y_1)
\end{pmatrix}
\]
yields the bound $|V_{P_1,P_2}(\F_q)|\geq q^4$.
\end{proof}

Now we can prove Theorem~\ref{main}.
\begin{proof}[Proof of Theorem~\ref{main}]
By Corollary~\ref{cor}, we have
\[
\Lambda_{P_1,P_2}(f_0,f_1,f_2)\ll_{P_1,P_2}\|f_0\|_{L^2}\|f_1\|_{L^2}\|f_2\|_{L^2}^{3/4}|\Lambda_{P_1,P_2}'(f_2,f_2)|^{1/8},
\]
and by Fourier inversion, Parseval's identity, and orthogonality of characters, we have
\[
\Lambda_{P_1,P_2}'(f_2,f_2)=\sum_{\psi\in\widehat{\F}_q}|\widehat{f_2}(\psi)|^2[\E_{y\in V_{P_1,P_2}(\F_q)}\psi(Q_{P_1,P_2}(y))].
\]
If $\psi=1$, then $\widehat{f_2}(\psi)=0$ since $f_2$ has mean zero, and if $\psi\neq 1$, then $\E_{y\in V_{P_1,P_2}(\F_q)}\psi(Q_{P_1,P_2}(y))\ll_{P_1,P_2}q^{-1/2}$ by Proposition~\ref{last}. Thus,
\[
|\Lambda_{P_1,P_2}'(f_2,f_2)|\ll_{P_1,P_2}\|f_2\|_{L^2}^2q^{-1/2}
\]
by Parseval's identity. We conclude that
\[
\Lambda_{P_1,P_2}(f_0,f_1,f_2)\ll_{P_1,P_2}\|f_0\|_{L^2}\|f_1\|_{L^2}\|f_2\|_{L^2}q^{-1/16}.
\]
\end{proof}
\bibliographystyle{plain}
\bibliography{bib}

\begin{thebibliography}{10}

\bibitem{BPPS}
A.~Balog, J.~Pelik\'an, J.~Pintz, and E.~Szemer\'edi.
\newblock Difference sets without {$\kappa$}th powers.
\newblock {\em Acta Math. Hungar.}, 65(2):165--187, 1994.

\bibitem{BL}
V.~Bergelson and A.~Leibman.
\newblock Polynomial extensions of van der {W}aerden's and {S}zemer\'edi's
  theorems.
\newblock {\em J. Amer. Math. Soc.}, 9(3):725--753, 1996.

\bibitem{BC}
J.~Bourgain and M.-C. Chang.
\newblock Nonlinear {R}oth type theorems in finite fields.
\newblock {\em Israel J. Math.}, Jul 2017.

\bibitem{BT}
B.~Bukh and J.~Tsimerman.
\newblock Sum-product estimates for rational functions.
\newblock {\em Proc. Lond. Math. Soc. (3)}, 104(1):1--26, 2012.

\bibitem{CLO}
D.~Cox, J.~Little, and D.~O'Shea.
\newblock {\em Ideals, varieties, and algorithms}.
\newblock Undergraduate Texts in Mathematics. Springer, New York, third
  edition, 2007.
\newblock An introduction to computational algebraic geometry and commutative
  algebra.

\bibitem{Go}
W.~T. Gowers.
\newblock A new proof of {S}zemer\'edi's theorem.
\newblock {\em Geom. Funct. Anal.}, 11(3):465--588, 2001.

\bibitem{HIS}
D.~Hart, A.~Iosevich, and J.~Solymosi.
\newblock Sum-product estimates in finite fields via {K}loosterman sums.
\newblock {\em Int. Math. Res. Not. IMRN}, (5):Art. ID rnm007, 14, 2007.

\bibitem{HLS}
D.~Hart, L.~Li, and C.-Y. Shen.
\newblock Fourier analysis and expanding phenomena in finite fields.
\newblock {\em Proc. Amer. Math. Soc.}, 141(2):461--473, 2013.

\bibitem{H}
R.~Hartshorne.
\newblock {\em Algebraic geometry}.
\newblock Springer-Verlag, New York-Heidelberg, 1977.
\newblock Graduate Texts in Mathematics, No. 52.

\bibitem{K}
G.~Kemper.
\newblock {\em A course in commutative algebra}, volume 256 of {\em Graduate
  Texts in Mathematics}.
\newblock Springer, Heidelberg, 2011.

\bibitem{Ko}
E.~Kowalski.
\newblock Exponential sums over definable subsets of finite fields.
\newblock {\em Israel J. Math.}, 160:219--251, 2007.

\bibitem{LW}
S.~Lang and A.~Weil.
\newblock Number of points of varieties in finite fields.
\newblock {\em Amer. J. Math.}, 76:819--827, 1954.

\bibitem{L}
J.~Lucier.
\newblock Intersective sets given by a polynomial.
\newblock {\em Acta Arith.}, 123(1):57--95, 2006.

\bibitem{P}
S.~Prendiville.
\newblock Quantitative bounds in the polynomial {S}zemer\'edi theorem: the
  homogeneous case.
\newblock {\em Discrete Anal.}, (5), 2017.

\bibitem{S}
A.~S\'ark\"ozy.
\newblock On difference sets of sequences of integers. {I}.
\newblock {\em Acta Math. Acad. Sci. Hungar.}, 31(1--2):125--149, 1978.

\bibitem{S2}
A.~S\'ark\"ozy.
\newblock On difference sets of sequences of integers. {III}.
\newblock {\em Acta Math. Acad. Sci. Hungar.}, 31:355--386, 1978.

\bibitem{Sh}
I.~D. Shkredov.
\newblock On monochromatic solutions of some nonlinear equations in {$\Bbb
  Z/p\Bbb Z$}.
\newblock {\em Mat. Zametki}, 88(4):625--634, 2010.

\bibitem{Sl}
S.~Slijep\v{c}evi\'c.
\newblock A polynomial {S}\'ark\"ozy-{F}urstenberg theorem with upper bounds.
\newblock {\em Acta Math. Hungar.}, 98(1-2):111--128, 2003.

\bibitem{T}
T.~Tao.
\newblock Expanding polynomials over finite fields of large characteristic, and
  a regularity lemma for definable sets.
\newblock {\em Contrib. Discrete Math.}, 10(1):22--98, 2015.

\bibitem{V}
V.~H. Vu.
\newblock Sum-product estimates via directed expanders.
\newblock {\em Math. Res. Lett.}, 15(2):375--388, 2008.

\end{thebibliography}

\end{document}